\documentclass[11pt]{article}
\usepackage{amsmath,amsfonts,amsthm,amssymb, mathtools}
\usepackage{mathrsfs, graphicx,color,latexsym, tikz, calc,
}
\usepackage{tikz-cd}
\usepackage{graphicx}
\usepackage{epstopdf}
\usepackage{enumerate}
\usepackage{caption}
\usepackage{stmaryrd}
\usepackage{hyperref}

\usetikzlibrary{shadows}
\usetikzlibrary{patterns,arrows,decorations.pathreplacing}
\usepackage{ulem}

\numberwithin{equation}{section}
\newtheorem{theorem}{Theorem}[section]
\newtheorem{lemma}[theorem]{Lemma}

\newtheorem{conjecture}[theorem]{Conjecture}

\newtheorem{claim}{Claim}[section]

\allowdisplaybreaks[4]

\date{\today}
\title{\bf \Large Improved bound on  symmetric differences of intersecting families\footnote{ Lihua Feng was supported by the NSFC (Nos. 12271527 and 12471022) and NSF of Qinghai Province (No. 2025-ZJ-902T). Zejun Huang was supported by the NSFC
(No. 12171323) and the Guangdong Basic and Applied Basic Research Foundation (No.
2022A1515011995)
E-mail addresses: \url{fenglh@163.com} (L. Feng),\url{zejunhuang@szu.edu.cn} (L. Feng), \url{jv01065499zai@163.com} (Q. Wang), \url{wuyjmath@163.com} (Y. Wu). }
\author{
{\small Lihua Feng$^{1}$, Zejun Huang$^{2}$, Qifan Wang$^{1}$\footnote{Corresponding author},\ \ Yongjiang Wu$^{1}$\ \ 
}\\[2mm]
\small $^1$School of Mathematics and Statistics, HNP-LAMA, Central South University\\
 \small Changsha, Hunan, 410083, China\\ 
 \small $^2$School of Mathematical Sciences, Shenzhen University\\ \small Shenzhen, 518060, China\\
}}

\begin{document}
	\maketitle
	\begin{abstract}
For a family  $\mathcal{F}$, it is called intersecting if $F\cap F'\neq \emptyset$ for all $F,F'\in\mathcal{F}$.
 We use $\mathcal{SD}(\mathcal{F}) = \{F \triangle G : F, G \in \mathcal{F}\}$ to denote
  the family of symmetric differences of $\mathcal{F}$. 
 In 2023, Frankl, Kiselev and Kupavskii conjectured that for any intersecting family $\mathcal{F} \subseteq \binom{[n]}{k}$ with $n > 10k$, the  inequality $|\mathcal{SD}(\mathcal{F})| \le \sum_{\ell=0}^{k-1} \binom{n-1}{2\ell}$ holds. They further observed that a proof for the range $n>3k^2$ could likely be obtained via arguments similar to those in their earlier work, though no detailed derivation was given.  In this paper, we establish the conjecture under the conditions  $n\ge 100k\ln k$ and  $k\ge 50$.
 We also  determine the extremal families, which are precisely a certain class of  stars.  A concentration inequality plays a central role in the proof.
\end{abstract}

	

\section{Introduction}

Let $[n]=\{1,2,\dots,n\}$ and let $\binom{[n]}{k}$  be the family of all $k$-element subsets of $[n]$. 
A family $\mathcal{F}\subseteq 2^{[n]}$ is called \textit{intersecting} if $F\cap F'\neq\emptyset$ for all $F,F'\in\mathcal{F}$. 
One of the cornerstones of extremal set theory is the Erd\H{o}s--Ko--Rado theorem \cite{E61}, which states that for $n\ge 2k$, any intersecting family $\mathcal{F}\subseteq\binom{[n]}{k}$ satisfies 
$|\mathcal{F}|\le\binom{n-1}{k-1}$, with equality only for \textit{ full star} 
$\mathcal{S}_x=\{F\in\binom{[n]}{k}:x\in F\}$ when $n>2k$. Subfamilies of $\mathcal{S}_x$ are called \textit{stars}.

Beyond the size, a natural direction is to study other set operations induced by an intersecting family. 
For a family $\mathcal{F}$, define its family of \textit{differences} as 
$\mathcal{D}(\mathcal{F})=\{F\setminus F':F,F'\in\mathcal{F}\}$. 
Marica and Sch\"onheim \cite{M69} proved the lower bound $|\mathcal{D}(\mathcal{F})|\ge|\mathcal{F}|$. 
For intersecting families, Frankl \cite{F21} conjectured an upper bound: for $n>2k$ and 
$\mathcal{F}\subseteq\binom{[n]}{k}$ intersecting,
$$
|\mathcal{D}(\mathcal{F})|\le\sum_{\ell=0}^{k-1}\binom{n-1}{\ell},
$$
with equality only for full stars. He proved this for $n\ge k(k+3)$. 
Later, Frankl, Kiselev and Kupavskii \cite{F23} significantly improved the range to 
$n\ge 50k\ln k$ for $k\ge 50$, and showed that the bound fails when $n<4k$ via explicit counterexamples.
Analogous problems for intersections $F \cap F'$ have also been studied extensively; see e.g. \cite{FKK21, FW23}.

A closely related operation is the \textit{symmetric difference} 
$F\triangle G=(F\cup G)\setminus(F\cap G)$. For a family $\mathcal{F}$, let
$$
\mathcal{SD}(\mathcal{F})=\{F\triangle G:F,G\in\mathcal{F}\}.
$$
For a full star $\mathcal{S}_x$, a direct calculation gives
$$
|\mathcal{SD}(\mathcal{S}_x)|=\sum_{\ell=0}^{k-1}\binom{n-1}{2\ell}.
$$
In the same work \cite{F23}, Frankl, Kiselev and Kupavskii proposed the following conjecture, 
which is the symmetric-difference analogue of the set-difference problem.

\begin{conjecture}[Frankl, Kiselev, Kupavskii \cite{F23}]
Let $\mathcal{F}\subseteq\binom{[n]}{k}$ be an intersecting family with $n>10k$. Then
$$
|\mathcal{SD}(\mathcal{F})|\le\sum_{\ell=0}^{k-1}\binom{n-1}{2\ell}. 
$$
\end{conjecture}

They remarked that, using arguments analogous to those in \cite{F21}, one could likely prove the conjecture for $n>3k^2$, though no detailed proof was given. Moreover, they emphasized that even establishing the conjecture 
for $n>100k\ln k$ would already be very interesting.
In this paper, we prove the conjecture for $n\ge 100k\ln k$ and $k\ge 50$, 
and show that equality holds only for  a certain class of  stars.

\begin{theorem}\label{M1}
Let $\mathcal{F}\subseteq\binom{[n]}{k}$ be an intersecting family with  $n\ge 100 k\ln k$ and $k\ge 50$. Then
$$
|\mathcal{SD}(\mathcal{F})|\le\sum_{\ell=0}^{k-1}\binom{n-1}{2\ell}, 
$$
where equality holds if and only if $\mathcal{F}\subseteq \mathcal{S}_x$ and $|\mathcal{SD}(\mathcal{F})|=|\mathcal{SD}(\mathcal{S}_x)|$ for some $x\in [n]$.
\end{theorem}

 Our proof is inspired by the elegant method of Frankl, Kiselev and Kupavskii \cite{F23} for set differences, but its adaptation to the symmetric difference setting is nontrivial and requires new ideas, particularly in the diversity analysis and the application of the concentration inequality. 
The proof proceeds in two stages. We first establish the conjecture for the range $n \geq 60k^{3/2}$ using the diversity method of   \cite{F23}, adapted to the symmetric difference setting. This already improves upon the previously suggested quadratic range $n > 3k^2$. To extend the result to the logarithmic range $n \geq 100k \ln k$, we refine the analysis of the non-star case: a sharper estimation of the error terms yields an exponential gain that suffices as long as $\frac{n}{k}$ is large enough relative to  $\ln k$. The concentration step, which is the main technical ingredient, requires only this logarithmic condition.

This paper is organized as follows. Section \ref{se2} collects the necessary notation and tools. In Section \ref{se3}, we  prove Theorem \ref{M1} for the range  $n\ge 60k^\frac{3}{2}$. In Section \ref{se4}, we handle the remaining  range $100 k\ln k \le n< 60k^\frac{3}{2}$, thereby completing the proof of Theorem \ref{M1}.

\section{Notation and tools} \label{se2}

For a family $\mathcal{G} \subseteq 2^{[n]}$ and an integer $\ell$, we write $\mathcal{G}^{(\ell)} = \{G \in \mathcal{G} : |G| = \ell\}$. 
For example,
$
\mathcal{SD}(\mathcal{F})^{(i)}= \{ D \in \mathcal{SD}(\mathcal{F}) : |D| = i \}.
$
For $A\subseteq B\subseteq [n]$, we use the standard notation $$\mathcal{G}(A,B)=\{G\setminus A:G\in \mathcal{G},\: G\cap B=A\}.$$
In the case $A=B$ or $A=\emptyset$, we use the shorter form $\mathcal{G}(B):=\mathcal{G}(B,B)$ and $\mathcal{G}(\overline{B}):=\mathcal{G}(\emptyset,B)$.
If $B=\{x\}$, then we set $\mathcal{G}(x):=\mathcal{G}(\{x\})$ and $\mathcal{G}(\overline{x}):=\mathcal{G}(\overline{\{x\}})$.

For a family $\mathcal{F} \subseteq \binom{[n]}{k}$ and  integers $p \le k \leq q$, we define the \textit{$p$-th level shadow} 
$$
\Delta^p \mathcal{F} = \left\{ P \in \binom{[n]}{p} : \exists F \in \mathcal{F}, \ P \subseteq F \right\},
$$
and the \textit{$q$-th level shade} 
$$\nabla^q\mathcal{F}=\{Q\in\binom{[n]}{q} :\exists F \in \mathcal{F}, \ F\subseteq Q\}.$$

The classical Kruskal--Katona theorem \cite{Kruskal1963, Katona1964} gives a lower bound on the size of the shadow. We will use the following convenient formulation due to Lov\'asz \cite{L79}.

\begin{theorem}\cite{L79}\label{L79}
Let $\mathcal{F} \subseteq \binom{[n]}{k}$ and write $|\mathcal{F}| = \binom{x}{k}$ for some real number $x$ with $k \le x \le n$. Then for any integer $p$ with $0 \le p \le k$,
$$
|\Delta^p \mathcal{F}| \ge \binom{x}{p}.
$$
\end{theorem}

Two families $\mathcal{F} \subseteq \binom{[n]}{k}$ and $\mathcal{G} \subseteq \binom{[n]}{\ell}$ are called \textit{cross intersecting} if $F \cap G \neq \emptyset$ for all $F \in \mathcal{F}, G \in \mathcal{G}$.  Katona \cite{Katona1966} observed that for $n \ge k + \ell$, $\mathcal{F}$ and $\mathcal{G}$ are cross-intersecting if and only if $\mathcal{F} \cap \Delta^k(\mathcal{G}^c) = \emptyset$, where $\mathcal{G}^c = \{[n] \setminus G : G \in \mathcal{G}\}$. Combining this observation with Theorem \ref{L79} yields the following bound.

\begin{theorem}\label{cros}
		For $n\ge k+\ell$, if $\mathcal{F}\subseteq\binom{[n]}{k}$ and $\mathcal{G}\subseteq \binom{[n]}{\ell}$ are cross-intersecting and $|\mathcal{F}|\ge \binom{x}{n-k}$ for some $n-k\leq x\leq n$, then $|\mathcal{G}|\le \binom{n}{\ell}-\binom{x}{\ell}$.
	\end{theorem}

Recall that $\mathcal{G}(\overline{H}) = \{ G \in \mathcal{G} : G \cap H = \emptyset \}$. 
The following concentration inequality is crucial for our proof. 

\begin{theorem}\cite{F23}\label{concen}
		Fix integers $m, \ell, \ell', t$ such that $m \geq t\ell + \ell'$, fix $a > 0$ and set $\varepsilon = \dfrac{2a+\sqrt{8\ln 2}}{\sqrt{t}}$. Let $\mathcal{G} \subseteq \binom{[m]}{\ell}$ be a family and set $\alpha := \dfrac{|\mathcal{G}|}{\binom{m}{\ell}}$. Let $H$ be distributed uniformly on $\binom{[m]}{\ell'}$. Then
		$$
		\mathbb P \left[ \mathcal{G}(\overline{H}) | < (\alpha - \varepsilon) \binom{m - \ell'}{\ell} \right] < 2e^{-a^2/2}. 
		$$
	\end{theorem}

	\section{The range $n\ge 60k^\frac{3}{2}$}\label{se3}
 
	\subsection{Reduction to two lemmas}

For a family $\mathcal{F} \subseteq \binom{[n]}{k}$, its \textit{diversity} is defined as
$$
\gamma(\mathcal{F}) = \min_{x \in [n]} |\{F \in \mathcal{F} : x \notin F\}|.
$$
Clearly $\gamma(\mathcal{F}) = 0$ if and only if $\mathcal{F}$ is a star. The concept of diversity was introduced by Lemons and Palmer \cite{L08} as a measure of how far an intersecting family deviates from a star. It has since then become a central tool in the study of large intersecting families. For any intersecting family $\mathcal{F} \subseteq \binom{[n]}{k}$,   Frankl and Wang \cite{F24} proved  that $\gamma(\mathcal{F}) \le \binom{n-3}{k-2}$ holds for $n > 36k$; this bound is currently the best known.  

 In our proof, we will not work directly with $\gamma(\mathcal{F})$ but with the diversity of its shade. Let $\mathcal{F}\subseteq\binom{[n]}{k}$ and $\mathcal{F}' = \nabla^{2k-1}\mathcal{F}$ be the $(2k-1)$-th level shade of $\mathcal{F}$. For $n \geq 2k$, one can check that $\gamma(\mathcal{F}) = 0$ if and only if $\gamma(\mathcal{F}') = 0$; and if $\gamma(\mathcal{F}) \ge 1$, then
$
\gamma(\mathcal{F}') \ge \binom{n-k-1}{k-1}.
$

The following two lemmas are the core of the proof of Theorem \ref{M1} for  $n\ge 60k^\frac{3}{2}$.

	\begin{lemma}\label{lem1}
		Let $\mathcal{F}\subseteq \binom{[n]}{k}$ be intersecting. Suppose that $n\ge 60k^\frac{3}{2}$, $k\ge 50$ and $|\mathcal{SD}(\mathcal{F})|\geq \sum\limits_{\ell=0}^{k-1}\binom{n-1}{2\ell}$. Then$$\gamma(\mathcal{F}')\le \binom{n-\sqrt{k}-1}{n-2k}.$$
		
	\end{lemma}
	\begin{lemma}\label{lem2}
		Let $\mathcal{F}\subseteq \binom{[n]}{k}$ be intersecting.  Suppose that $n\ge 60k^\frac{3}{2}$, $k\ge 50$ and $1\le \gamma(\mathcal{F}')\le \binom{n-\sqrt{k}}{n-2k}$. Then $$|\mathcal{SD}(\mathcal{F})|< \sum\limits_{\ell=0}^{k-1}\binom{n-1}{2\ell}.$$
	\end{lemma}

For  $n\ge 60k^\frac{3}{2}$,  Theorem \ref{M1} follows directly from the two lemmas above. 
 If $\gamma(\mathcal{F}) = 0$, then $\mathcal{F}$ is a subfamily of a full star and Theorem \ref{M1} holds.  If $\gamma(\mathcal{F}) \ge 1$, then $\gamma(\mathcal{F}') \ge \binom{n-k-1}{k-1} > 0$. 
Suppose, for contradiction, that  $|\mathcal{SD}(\mathcal{F})|\geq \sum\limits_{\ell=0}^{k-1}\binom{n-1}{2\ell}$. Then Lemma \ref{lem1} implies $\gamma(\mathcal{F}') \leq \binom{n-\sqrt{k}-1}{n-2k}\leq \binom{n-\sqrt{k}}{n-2k}$, contradicting Lemma \ref{lem2}. Thus, Theorem \ref{M1} holds in the range $n\ge 60k^\frac{3}{2}$.

	\subsection{Proof of Lemma \ref{lem1}}

	Let $\mathcal{F}\subseteq \binom{[n]}{k}$ be intersecting and $|\mathcal{SD}(\mathcal{F})|\geq \sum\limits_{\ell=0}^{k-1}\binom{n-1}{2\ell}.$ Denote $C:=\frac{n}{k}$, $\mathcal{F}' = \nabla^{2k-1} \mathcal{F}$. \\

    {\bf We first bound $ \left| \mathcal{SD}(\mathcal{F})^{(2k-2)} \right|$ from below}. 
    \\ 
    
    For $1 \le i \le 2k-5$, we have
	$$\frac{\binom{n-1}{i-1}}{\binom{n-1}{i}}=\frac{i}{n-i}\le\frac{2k}{n - 2k} < \frac{1}{2}.$$
	It follows that 
    \begin{align*}
    \sum_{i=1}^{k-2} \binom{n-1}{2i-1} <& 2\binom{n-1}{2k-5} \\
= & \frac{2(2k-2)(2k-3)(2k-4)}{(n-2k+2)(n-2k+3)(n-2k+4)} \binom{n-1}{2k-2}\\
    < &\frac{16}{(C-2)^3} \binom{n-1}{2k-2}.
    \end{align*}
Since $|\mathcal{SD}(\mathcal{F})|\geq \sum\limits_{l=0}^{k-1}\binom{n-1}{2l}$, we obtain
	\begin{align}
    \left| \mathcal{SD}(\mathcal{F})^{(2k-2)} \right|\geq &\sum_{i=0}^{k-1} \binom{n-1}{2i} - \sum_{i=0}^{k-2} \binom{n}{2i} \notag\\
    =& \binom{n-1}{2k-2}- \sum_{i=1}^{k-2} \binom{n-1}{2i-1} \notag\\
    \geq &\left( 1 - \frac{16}{(C-2)^3} \right) \binom{n-1}{2k-2}.\label{fo1}
	\end{align}

{\bf Next we relate $\mathcal{SD}(\mathcal{F})^{(2k-2)}$ to $\mathcal{F}'$}.  \\ 

Every symmetric difference of size $2k-2$ is contained in some member of $\mathcal{F}'$. Moreover, each member of $\mathcal{F}'$ can contain at most $2k-1$ such symmetric differences. Therefore,
\begin{align}
|\mathcal{SD}(\mathcal{F})^{(2k-2)}| \le (2k-1)|\mathcal{F}'| \le 2k|\mathcal{F}'|.\label{fo2}
\end{align}
For $C\ge6$, we derive from (\ref{fo1}) and  (\ref{fo2}) that
 \begin{align}|\mathcal{F}'| \geq\frac{3}{8k} \binom{n-1}{2k-2} = \frac{3(2k-1)}{8nk} \binom{n}{2k-1} \geq \frac{1}{2n} \binom{n}{2k-1}.\label{fo3}
 \end{align}
By (\ref{fo1}), the density of $\mathcal{SD}(\mathcal{F})^{(2k-2)}$ satisfies
\begin{align*}
 \frac{|\mathcal{SD}(\mathcal{F})^{(2k-2)}|}{\binom{n}{2k-2}} &\geq \left(1 - \frac{16}{(C-2)^3}\right) \frac{n-2k+2}{n} \\
 &\geq \left(1 - \frac{16}{(C-2)^3}\right) \left(1 - \frac{2}{C}\right)\\
 & \geq 1 - \frac{3}{C}.
\end{align*}

{\bf We now apply the concentration inequality}.\\

	Apply Theorem \ref{concen} to $\mathcal{SD}(\mathcal{F})^{(2k-2)}$ with $m=n$, $\ell=2k-2$, $\ell'=2k-1$, $t=\lfloor \frac{C}{2} \rfloor-1$, $a=\sqrt{2\ln 8n}$, $\varepsilon=\frac{2a+\sqrt{8\ln 2}}{\sqrt{t}}$, $\alpha=\frac{|\mathcal{SD}(\mathcal{F})^{(2k-2)}|}{\binom{n}{2k-2}}\ge 1-\frac{3}{C}$.
A direct computation shows
\begin{align*}
 \varepsilon &\le \frac{2\sqrt{2\ln 8Ck} + 2.4}{\sqrt{\frac{C}{2}-2}} 
 < \frac{2\sqrt{2\ln(480k^\frac{3}{2})} + 2.4}{\sqrt{30\sqrt{k}-2} }\\
  &< 2\frac{\sqrt{2\ln(480)+3\ln{k}}}{\sqrt{29\sqrt{k}}} + 0.17 
  < 0.86 < 0.9 - \frac{3}{C}.
\end{align*}
	Then $\alpha-\varepsilon>\frac{1}{10}$.
    Let $H$ be distributed uniformly on $\binom{[n]}{2k-1}$. 
By (\ref{fo3}), we have $\mathbb P\left[H\in \mathcal{F}'\right]\ge \frac{1}{2n}$.
Moreover, Theorem \ref{concen} gives
   $$\mathbb P\left[\left|\mathcal{SD}(\mathcal{F})^{(2k-2)}\cap \binom{\overline{H}}{2k-2}\right|< \frac{1}{10}\binom{n-2k+1}{2k-2}\right]< 2e^{-\frac{\mathstrut{a^2}}{2}}=\frac{1}{4n},$$
   where $\overline{H}=[n]\backslash H$.
	Therefore,  there exists a set $F'\in \mathcal{F}'$ such that $\left|\mathcal{SD}(\mathcal{F})^{(2k-2)}\cap\binom{\overline{F'}}{2k-2}\right|\ge\frac{1}{10}\binom{n-2k+1}{2k-2}$. Without loss of generality,  assume $F'=[2k-1]$.\\

{\bf Finally, we analyze the structure of symmetric differences disjoint from $F'$}.\\

	For each $S\in \mathcal{SD}(\mathcal{F})^{(2k-2)}\cap\binom{\overline{F'}}{2k-2}$, there exists an element  $x\in [n]\backslash S$ such that $S\cup \{x\}\in \mathcal{F}'$. Since $\mathcal{F}'$ is intersecting, we have  $\left( S\cup \{x\} \right)\cap F'\neq \emptyset$. Thus, $x\in F'$.  Split the family $\mathcal{SD}(\mathcal{F})^{(2k-2)} \cap \binom{\overline{F'}}{2k-2}$ into $\bigcup_{x \in F'} \mathcal{E}_x$ accordingly, where $S \in \mathcal{E}_x$ implies $\{x\} \cup S \in \mathcal{F}'$. The families $\mathcal{E}_x$, $x \in [2k-1]$ are pairwise cross intersecting.
	Let $m=n-2k+1$, $\ell=2k-2$.

	\begin{claim}\label{cl1}
		For some $i\in [2k-1]$, we have $|\mathcal{E}_i| > \binom{m}{\ell} - \binom{m-\sqrt{k}}{\ell}$.
	\end{claim}
	\begin{proof}
For each $i\in [2k-1]$, consider the inequality
$|\mathcal{E}_i| < |\mathcal{SD}(\mathcal{F'})^{(2k-2)} \cap \binom{\overline{F'}}{2k-2}| - \binom{m-1}{\ell-1}$.
If this inequality holds for all  $i\in [2k-1]$,
	 then  for every $i$ we have $\sum_{j\neq i}|\mathcal{E}_j|> \binom{m-1}{\ell-1}$.
     We claim that  one can partition $[2k-1]$ into two disjoint nonempty subsets $I_1$ and $ I_2$ such that
		 \begin{align}
		\Bigl| \bigcup_{i \in I_1} \mathcal{E}_i \Bigr| > \binom{m-1}{\ell-1}
		\quad\text{and}\quad
		\Bigl| \bigcup_{i \in I_2} \mathcal{E}_i \Bigr| > \binom{m-1}{\ell-1}.\label{fo4}
\end{align}
Indeed, if there exists an index $j$ such that $|\mathcal{E}_j|>\binom{m-1}{\ell-1}$, we may simply take $I_1=\{j\}$ and $I_2=[2k-1]\backslash \{j\}$.
Otherwise, assume $|\mathcal{E}_i|\leq\binom{m-1}{\ell-1}$ for all $i$.
Since $ \sum_{j\neq i}|\mathcal{E}_j|> \binom{m-1}{\ell-1}$ for every $i$, we can choose the smallest integer $t\geq 2$  such that  $ \sum_{j=1}^{t} |\mathcal{E}_j|> \binom{m-1}{\ell-1}$.
Then $ \sum_{j=1}^{t-1} |\mathcal{E}_j|\leq \binom{m-1}{\ell-1}$.
Consequently,   $\sum_{j=t+1}^{2k-1} |\mathcal{E}_j|\geq  \sum_{j=1}^{2k-1} |\mathcal{E}_j|-2\binom{m-1}{\ell-1}$.
Using the lower bound $\sum_{j=1}^{2k-1} |\mathcal{E}_j|=\left|\mathcal{SD}(\mathcal{F})^{(2k-2)}\cap\binom{\overline{F'}}{2k-2}\right|\ge\frac{1}{10}\binom{n-2k+1}{2k-2}$, we obtain
$$\sum_{j=t+1}^{2k-1} |\mathcal{E}_j|\ge\frac{1}{10}\binom{m}{\ell}-2\binom{m-1}{\ell-1}>\binom{m-1}{\ell-1}.$$
Thus take  $I_1=[t]$ and $I_2=[t+1, 2k-1]$ satisfies (\ref{fo4}).
Now observe that  $\bigcup_{i \in I_1} \mathcal{E}_i$ and   $\bigcup_{i \in I_2} \mathcal{E}_i$ are cross intersecting. Condition (\ref{fo4}) contradicts   Theorem \ref{cros}.
		
		Thus, there exists some $i$ such that $|\mathcal{E}_i| \geq |\mathcal{SD}(\mathcal{F})^{(2k-2)} \cap \binom{\overline{F'}}{k-1}| - \binom{m-1}{\ell-1}$. 
		It suffices to verify that the right-hand side is at least $\binom{m}{\ell} - \binom{m-\sqrt{k}}{\ell}+1$. First, note that
		$\frac{\ell}{m}=\frac{2k-2}{n-2k+1}\le\frac{2}{C-2}<\frac{1}{60}$.
         It follows that
		$$
		|\mathcal{E}_i| \geq \biggl( \frac{1}{10} - \frac{\ell}{m} \biggr) \binom{m}{\ell} > \frac{1}{12} \binom{m}{\ell}.
		$$
		It remains to show that $\binom{m-\sqrt{k}}{\ell} > \frac{11}{12} \binom{m}{\ell}$. Observe that
		\begin{align*}
			m - \ell - \sqrt{k} + 1 &= n-2k+1-2k+2-\sqrt{k}+1\\
&>60k\sqrt{k}-5k>50k\sqrt{k}>20\ell \sqrt{k}. 
		\end{align*}
		Therefore, we have
		\begin{align*}
			\frac{\binom{m}{\ell}}{\binom{m-\sqrt{k}}{\ell}}&=\prod_{i=1}^{l} \frac{m-i+1}{m-\sqrt{k}-i+1}
\le \left(\frac{m-\ell+1}{m-\sqrt{k}-\ell+1}\right)^\ell
= \left(1+\frac{\sqrt{k}}{m - \sqrt{k} - \ell + 1}\right)^\ell \\
&\le \exp\biggl( \frac{\ell \sqrt{k}}{m - \sqrt{k} - \ell + 1} \biggr)
\le e^{\frac{1}{20}}<\frac{12}{11}.
		\end{align*}
		This completes the proof of the claim.
	\end{proof}

Without loss of generality, we assume that the index $i$ guaranteed by Claim \ref{cl1} is $1$.

\begin{claim}\label{cl2}
Let $A \subseteq F'$ be such that $1 \notin A$. Then
\[
|\mathcal{F}'(A, F')| \le \binom{m - \sqrt{k}}{m - 2k + 1 + |A|}.
\]
\end{claim}

\begin{proof}
Note that $\mathcal{F}'(A, F') \subseteq \binom{[n] \setminus F'}{2k-1-|A|}$. For any $H \in \mathcal{E}_1$ and any $T \in \mathcal{F}'(A, F')$, the sets $H \cup \{1\}$ and $T \cup A$ belong to $\mathcal{F}'$. Since $\mathcal{F}'$ is intersecting and $1 \notin A, T$ and $S\subseteq \overline{F'}$, we have $H\cap T\neq \emptyset$.
Hence, the families $\mathcal{E}_1$ and $\mathcal{F}'(A, F')$ are cross intersecting. Applying Theorem \ref{cros} together with the bound $|\mathcal{E}_1| > \binom{m}{\ell} - \binom{m - \sqrt{k}}{\ell}$ yields the desired inequality.
\end{proof}

\begin{claim}\label{cl3}
$
|\mathcal{F}'(\overline{1})| \le \binom{n - \sqrt{k} - 1}{n - 2k}.
$
\end{claim}

\begin{proof}
Decomposing $|\mathcal{F}'(\overline{1})|$ according to the intersection with $F'$ and applying Claim \ref{cl2}, we obtain
$$
\begin{aligned}
|\mathcal{F}'(\overline{1})| &= \sum_{\substack{A \subseteq F' , 1 \notin A}} |\mathcal{F}'(A, F')| \\
&\le \sum_{a=0}^{2k-2} \binom{2k-2}{2k-2-a} \binom{n - 2k + 1 - \sqrt{k}}{n - 2k + 1 - 2k + 1 + a} \\
&= \sum_{a=0}^{2k-2} \binom{2k-2}{2k-2-a} \binom{n - \sqrt{k} - 1 - (2k-2)}{n - 2k - (2k-2 - a)} \\
&= \binom{n - \sqrt{k} - 1}{n - 2k}.
\end{aligned}
$$
Thus the claim holds.
\end{proof}

Lemma \ref{lem1} now follows immediately from Claim  \ref{cl3}.

	\subsection{Proof of Lemma \ref{lem2}}
    Let  $\mathcal{F}\subseteq \binom{[n]}{k}$ be intersecting  and $\mathcal{F}' = \nabla^{2k-1}\mathcal{F}$.
	Suppose that $\gamma(\mathcal{F}') = |\mathcal{F}'(\overline{1})|$ satisfies
$
1 \le \gamma(\mathcal{F}') \le \binom{n - \sqrt{k}}{n - 2k}.
$
Then $|\gamma(\mathcal{F})|\ge 1$, and consequently
$\gamma(\mathcal{F}')\ge\binom{n-k-1}{k-1}=\binom{n-k-1}{n-2k}$.
Assume that $\gamma({\mathcal{F}'})=|\mathcal{F}'(\overline1)|=\binom{x}{n-2k}$ for some $x$ with $n-k-1\le x\le n-\sqrt{k}.$
Define
$$
\mathcal{G}' = \{[2,n] \setminus F' : F' \in \mathcal{F}'(\overline{1})\} \subseteq \binom{[2,n]}{n-2k},
$$
and $\mathcal{G} = \{[2,n] \setminus F : F \in \mathcal{F}(\overline{1})\}$. 
For  $0\le t\le k-1$, define
 $$\mathcal{G}^{k+t}=\{[2, n] \setminus T : T \in (\nabla^{k+t}\mathcal{F})(\overline{1})\}.$$ 
Observe that $(\nabla^{k+t}\mathcal{F})(\overline{1})=\{P\in\binom{[n]}{k+t} :\exists F \in \mathcal{F}, \ F\subseteq P,\ 1\notin P\}=\{P\in\binom{[2, n]}{k+t} :\exists F \in \mathcal{F}(\overline{1}), \ F\subseteq P\}=\nabla^{k+t}(\mathcal{F}(\overline{1}))
$.
Moreover,  $\mathcal{G}^{k+0}=\mathcal{G}$ and $\mathcal{G}^{k+k-1}=\mathcal{G}'$.
	  Since $\mathcal{F}'=\nabla^{2k-1}\mathcal{F}$, we have $\mathcal{G}'=\Delta^{n-2k}{\mathcal{G}^{k+t}}$ for every $0 \leq t\leq k-1$.
      Assume that $|\mathcal{G}^{k+t}|=\binom{y}{n-1-k-t}$. By Theorem \ref{L79}, we obtain $|\mathcal{G}'|\geq \binom{y}{n-2k}$.
      Since $|\mathcal{G}'|=|\mathcal{F}'(\overline1)|=\binom{x}{n-2k} $, it follows that $y\leq x$.
Consequently,  $$|\mathcal{G}^{k+t}|\le \binom{x}{n-k-t-1}.$$ Additionally, since  $n\ge 4k$ and $\mathcal{G}'=\Delta^{n-2k}{\mathcal{G}^{k+0}}=\Delta^{n-2k}{\mathcal{G}}$,  we have $\Delta^i{\mathcal{G}'}=\Delta^i{\mathcal{G}}$ for $1\le i \le 2k$.
	
		\begin{claim}\label{cl4}
	For $0 \le t \le k-1$,	$\frac{|\mathcal{G}^{k+t}|}{|\Delta^{k+t-1}{\mathcal{G}}|}\le e^{-\sqrt{k}}.$ 
		
	\end{claim}
	\begin{proof}
Since  $\Delta^{k+t-1}{\mathcal{G}}=\Delta^{k+t-1}{\mathcal{G}'}$ and $|\mathcal{G}'|=\binom{x}{n-2k} $, 
  it follows from  Theorem \ref{L79} that
  $|\Delta^{k+t-1}{\mathcal{G}}|\geq\binom{x}{k+t-1}$.
Therefore, we have
\begin{align*} 
\frac{|\mathcal{G}^{k+t}|}{|\Delta^{k+t-1}{\mathcal{G}}|}&\le \frac{\binom{x}{n-k-t-1}}{\binom{x}{k+t-1}}
= \prod\limits_{i=k+t}^{n-k-t-1} \frac{x-i+1}{i}\\
&\le \prod\limits_{i=k+t}^{n-k-t-1} \frac{n-\sqrt{k}-i+1}{i}
=\prod\limits_{i=k+t}^{n-k-t-1} \frac{i-\sqrt{k}+2}{i}
\\
&
=\prod\limits_{i=k+t}^{n-k-t-1} \left(1-\frac{\sqrt{k}-2}{i}\right)\\
  &\le e^{-(\sqrt{k}-2)\sum\limits_{i=k+t}^{n-k-t-1}\frac{1}{i}}\\
  &\le e^{-(\sqrt{k}-2)\ln \frac{n-k-t}{k+t}} 
  \le e^{-\sqrt{k}},
\end{align*}
  where the last inequality follows from 
	 $\sqrt{k}-2\ge \frac{\sqrt{k}}{2}$ and $\ln \left(\frac{n-k-t}{k+t}\right)\ge \ln \left(\frac{n-2k}{2k}\right)\ge 2$.
	\end{proof}

    For convenience, we write $\mathcal{SD}=\mathcal{SD(F)}$.
	Observe that $|\mathcal{SD}|=|\mathcal{SD}(1)|+|\mathcal{SD}(\overline{1})|$. We first consider $\mathcal{SD}(\overline{1})$.
	For $0\le i\le k-2$, we have 
     \begin{align*}
    \sum\limits_{i=0}^{k-2} |\mathcal{SD}^{(2i)}(\overline{1})| \leq \sum\limits_{i=0}^{k-2} \binom{n-1}{2i}. 
    \end{align*}
 The remaining part $\mathcal{SD}^{(2k-2)}(\overline{1})$ may arise from symmetric differences between two sets both in $\mathcal{F}(1)$, or both in $\mathcal{F}(\overline{1})$.
	By the same reasoning as in  (\ref{fo2}), we have  	$|\mathcal{SD}(\mathcal{F}(\overline{1}))^{(2k-2)}|\le (2k-1)|\mathcal{F}'(\overline{1})|$.
    Applying Claim \ref{cl4} yields
   $$|\mathcal{SD}(\mathcal{F}(\overline{1}))^{(2k-2)}| \le 2k|\mathcal{G}'|\le \frac{2k}{e^{\sqrt{k}}}|\Delta^{2k-2}{\mathcal{G}}|\le \frac{1}{3}|\Delta^{2k-2}{\mathcal{G}}|.$$
		For any $H\in \mathcal{SD}(\mathcal{F}(1))^{(2k-2)}$, we have ${\{1}\} \cup H \in \mathcal{F}'$. If $H\in \Delta^{2k-2}{\mathcal{G}}$, then there exists $F' \in \mathcal{F}'(\overline{1})$ such that $H\cap F'=\emptyset$. This would imply  $({\{1}\} \cup H ) \cap  F'=\emptyset$, contradicting the intersecting property of $\mathcal{F}'$.
        Hence, $H\notin \Delta^{2k-2}{\mathcal{G}}$.
        Consequently, 
	$$|\mathcal{SD}(\mathcal{F}(1))^{(2k-2)}|\le \binom{n-1}{2k-2}-|\Delta^{2k-2}{\mathcal{G}}|.$$
Therefore,
\begin{align*}
|\mathcal{SD}^{(2k-2)}(\overline{1})|&=|\mathcal{SD}(\mathcal{F}(\overline{1}))^{(2k-2)}|+|\mathcal{SD}(\mathcal{F}(1))^{(2k-2)}|\\
&\le \binom{n-1}{2k-2}-|\Delta^{2k-2}{\mathcal{G}}| +\frac{1}{3} |\Delta^{2k-2} \mathcal{G}|.
\end{align*}
Since  $|\mathcal{SD}(\overline{1})|= \sum\limits_{i=0}^{k-2} |\mathcal{SD}^{(2i)}(\overline{1})| +|\mathcal{SD}^{(2k-2)}(\overline{1})|$, we obtain
 \begin{align}
   |\mathcal{SD}(\overline{1})|\leq \sum\limits_{i=0}^{k-2} \binom{n-1}{2i}+\binom{n-1}{2k-2}-|\Delta^{2k-2}{\mathcal{G}}| +\frac{1}{3} |\Delta^{2k-2} \mathcal{G}|. \label{fo5}
    \end{align}

We now turn to $\mathcal{SD}(1)$. Any set in $\mathcal{SD}(1)$ must be the symmetric difference of one set from $\mathcal{F}(1)$ and one from $\mathcal{F}(\overline{1})$.
Hence, for any $D\in \mathcal{SD}(1)$,  we have $|D|=2i+1$ with $0\le i\le k-2$. 
For any $D \in \mathcal{SD}(1)^{(2i+1)}$, there exist $A \in \mathcal{F}(1)$ and $B \in \mathcal{F}(\overline{1})$ such that $D = A \triangle B$. Then  $|A\cup B|=k+i$. Thus, $D$ is contained in some member of $\nabla^{k+i}(\mathcal{F}(\overline{1}))=\{Q\in\binom{[2,n]}{k+i} :\exists F \in \mathcal{F}(\overline{1}), \ F\subseteq Q\}$.  Each member of $\nabla^{k+i}(\mathcal{F}(\overline{1}))$ may contain at most $\binom{k+i}{2i+1}$ such symmetric differences. 
Therefore, 
 \begin{align}
 |\mathcal{SD}(1)^{(2i+1)}|\le \binom{k+i}{2i+1}|\nabla^{k+i}(\mathcal{F}(\overline{1}))|=\binom{k+i}{k-i-1}|\nabla^{k+i}(\mathcal{F}(\overline{1}))|.\label{fo6}
 \end{align}
 
Recall that  $\mathcal{G} = \{[2,n] \setminus F : F \in \mathcal{F}(\overline{1})\}$.

\begin{claim}\label{cl6}
For $1\le i\le2k-2$,	we have $|\Delta^{i-1}(\mathcal{G})|\le \frac{1}{4\sqrt{k}}|\Delta^{i}(\mathcal{G})|$. In particular,  for $0\le t\le k-1$,  $$|\Delta^{k+t-1} \mathcal{G}| \leq \left(\frac{1}{4\sqrt{k}}\right)^{k-t-1} |\Delta^{2k-2} \mathcal{G}|.$$
\end{claim}

\begin{proof}
Consider the bipartite graph between $\Delta^{i-1}\mathcal{G}$ and $\Delta^{i}\mathcal{G}$, where  $A\in \Delta^{i-1}\mathcal{G}$ and $B\in \Delta^{i}\mathcal{G}$ are adjacent if $A \subseteq B$. Let $e$ denote the number of edges.
On the one hand, each set $B \in \Delta^{i}\mathcal{G}$ contains exactly $i$ subsets of size $i-1$. Thus, $e = i |\Delta^{i}\mathcal{G}|$.
On the other hand, for any $A \in \Delta^{i-1}\mathcal{G}$, there exists some $G \in \mathcal{G}$ with $A \subseteq G$. The number of $i$-sets $B$ such that $A \subseteq B \subseteq G$ is at least $n-k-i$. Hence, each $A$ contributes at least $n-k-i$ edges, giving $e \ge (n-k-i) |\Delta^{i-1}\mathcal{G}|$.
Combining the two estimates yields
$
i |\Delta^{i}\mathcal{G}| \ge (n-k-i) |\Delta^{i-1}\mathcal{G}|.
$
Since $n-k-i \ge 60k\sqrt{k} - 3k \ge 8k\sqrt{k}$ and $i \le 2k-2$, we obtain
$$
|\Delta^{i-1}\mathcal{G}| \le \frac{i}{n-k-i} |\Delta^{i}\mathcal{G}| \le \frac{2k}{8k\sqrt{k}} |\Delta^{i}\mathcal{G}| = \frac{1}{4\sqrt{k}} |\Delta^{i}\mathcal{G}|.
$$
Repeated application of this inequality gives the second bound of the claim.		
\end{proof}

Applying  Claims \ref{cl4} and  \ref{cl6} together with (\ref{fo6}), we obtain
\begin{align*}
|\mathcal{SD}(1)|&\le \sum\limits_{i=0}^{k-2}|\mathcal{SD}(1)^{(2i+1)}|
\le \sum\limits_{i=0}^{k-2}\binom{k+i}{k-i-1}|\nabla^{k+i}(\mathcal{F}(\overline{1}))|\\
&=\sum\limits_{i=0}^{k-2}\binom{k+i}{k-i-1}|\mathcal{G}^{k+i}|
\le\sum\limits_{i=0}^{k-2}\frac{1}{e^{\sqrt{k}}}|\Delta^{k+i-1}(\mathcal{G})|\binom{2k}{k-i-1}\\
&\le \sum\limits_{i=0}^{k-2}\frac{1}{e^{\sqrt{k}}}\left( \frac{1}{4\sqrt{k}} \right)^{k-i-1}|\Delta^{2k-2} \mathcal{G}|\binom{2k}{k-i-1}\\
&\le \frac{|\Delta^{2k-2} \mathcal{G}|}{e^{\sqrt{k}}}\sum\limits_{i=1}^{k-1}{\left( \frac{1}{4\sqrt{k}}\right)}^i\binom{2k}{i}\\
&\le\frac{|\Delta^{2k-2} \mathcal{G}|}{e^{\sqrt{k}}}\sum\limits_{i=0}^{2k}{\left( \frac{1}{4\sqrt{k}}\right)}^i\binom{2k}{i}\\
	&=\frac{|\Delta^{2k-2} \mathcal{G}|}{e^{\sqrt{k}}}\left(1+\frac{1}{4\sqrt{k}} \right)^{2k}\\
&\le\frac{|\Delta^{2k-2} \mathcal{G}|}{e^{\sqrt{k}}}e^{\frac{2k}{4\sqrt{k}}}
= \frac{|\Delta^{2k-2} \mathcal{G}|}{e^{\frac{\sqrt{k}}{2}}}
<\frac{1}{6}|\Delta^{2k-2} \mathcal{G}|.
\end{align*}
Combining this with (\ref{fo5}), we get   
\begin{align*}|\mathcal{SD}|&=|\mathcal{SD}(1)|+|\mathcal{SD}(\overline{1})|\\
&<\sum\limits_{i=0}^{k-2} \binom{n-1}{2i}+\binom{n-1}{2k-2}-|\Delta^{2k-2}{\mathcal{G}}| +\frac{1}{3} |\Delta^{2k-2} \mathcal{G}|+\frac{1}{6} |\Delta^{2k-2} \mathcal{G}|\\
&= \sum_{i=0}^{k-1} \binom{n-1}{2i} - \frac{1}{2} |\Delta^{(2k-2)}\mathcal{G}<  \sum_{i=0}^{k-1} \binom{n-1}{2i}.
\end{align*}
This completes the proof of Lemma \ref{lem2}.

\section{The range $100 k\ln k \le n< 60k^\frac{3}{2}$}\label{se4}

\subsection{Reduction to two lemmas and elementary estimates}
 Let  $\mathcal{F}\subseteq \binom{[n]}{k}$ be intersecting  and let $\mathcal{F}' = \nabla^{2k-1}\mathcal{F}=\{Q\in\binom{[n]}{2k-1} :\exists F \in \mathcal{F}, \ F\subseteq Q\}$. We shall need the following two lemmas.

\begin{lemma}\label{lem3}
		Let $\mathcal{F}\subseteq \binom{[n]}{k}$ be intersecting. Assume that  $k\ge 50$ and $100 k\ln k \le n< 60k^\frac{3}{2}$, and let $
r = \left\lfloor \frac{n}{30k} \right\rfloor$. If $|\mathcal{SD}(\mathcal{F})|\geq \sum\limits_{\ell=0}^{k-1}\binom{n-1}{2\ell}$,
then 
$$\gamma(\mathcal{F}')\le \binom{n-r-1}{n-2k}.$$
	
	\end{lemma}
	\begin{lemma}\label{lem4}
		Let $\mathcal{F}\subseteq \binom{[n]}{k}$ be intersecting.  Assume that  $k\ge 50$ and $100 k\ln k \le n< 60k^\frac{3}{2}$,  and let $
r = \left\lfloor \frac{n}{30k} \right\rfloor$.
If 
$1\le \gamma(\mathcal{F}')\le \binom{n-r}{n-2k}$, then 
$$
|\mathcal{SD}(\mathcal{F})|< \sum\limits_{\ell=0}^{k-1}\binom{n-1}{2\ell}.
$$
	\end{lemma}

For  $100 k\ln k \le n< 60k^\frac{3}{2}$,  Theorem \ref{M1} follows directly from the two lemmas above. 
 If $\gamma(\mathcal{F}) = 0$, then $\mathcal{F}$ is a  star and Theorem \ref{M1} holds.  If $\gamma(\mathcal{F}) \ge 1$, then $\gamma(\mathcal{F}') \ge \binom{n-k-1}{k-1} > 0$. 
Suppose, for contradiction, that  $|\mathcal{SD}(\mathcal{F})|\geq \sum\limits_{\ell=0}^{k-1}\binom{n-1}{2\ell}$. Applying Lemma \ref{lem3} yields $\gamma(\mathcal{F}') \leq \binom{n-r-1}{n-2k}\leq \binom{n-r}{n-2k}$. By Lemma \ref{lem4}, we have $
|\mathcal{SD}(\mathcal{F})|< \sum\limits_{\ell=0}^{k-1}\binom{n-1}{2\ell}
$,  a contradiction. Therefore, Theorem \ref{M1} holds in the range $100 k\ln k \le n< 60k^\frac{3}{2}$.

It remains to prove Lemmas \ref{lem3} and \ref{lem4}. The following elementary estimates will be used in both proofs.
For the rest of the proof, set  $C = \frac{n}{k}$ and $L=\ln k$.
Then $100 \ln k\leq  C<60\sqrt{k}$.
Let $
r = \left\lfloor \frac{C}{30} \right\rfloor$.
Since $k\geq 50$, we have $L>3.9, C>300$ and $r>\frac{C}{40}$.

\begin{lemma}\label{le5}
Assume that $k\geq 50$ and
$
100\ln k \le C < 60\sqrt{k}.
$
Let
$
q = \frac{3}{C}$ and $M = \ln \frac{C}{3}.
$
Then $M > 4$, and the following estimates hold.

\begin{enumerate}[(1)]
\item If
$
a = \sqrt{2\ln(8n)}$ and $t = \left\lfloor \frac{C}{2} \right\rfloor - 1,
$
then
\[
\frac{2a + \sqrt{8\ln 2}}{\sqrt{t}} < 0.88.
\]

\item For every real $h$ satisfying $r \le h \le k+1$, one has
\[
2k q^{h-2} < \frac{1}{8}.
\]

\item For every real $h$ satisfying $r \le h \le k+1$, one has
\[
\sum_{j=1}^{k-1} \min\left\{ e^{4h/C} q^{2j-1}, \binom{2k}{j} q^{h+j-2} \right\} < \frac{1}{8}.
\]
\end{enumerate}
\end{lemma}

\begin{proof}
First, we have
$
C \ge 100\ln k \ge 100\ln 50 > 300.
$
Hence,
$
M = \ln(C/3) > \ln 100 > 4.
$

For (1), since $C < 60\sqrt{k}$, we have
$
8n = 8Ck < 480k^{3/2}.
$
Also,
$
t = \left\lfloor \frac{C}{2} \right\rfloor - 1 \ge \frac{C}{2} - 2 \ge 50L - 2,
$
where $L = \ln k$. 
Thus,
$$
\frac{2a + \sqrt{8\ln 2}}{\sqrt{t}}
\le \frac{2\sqrt{2\ln(480k^{\frac{3}{2}})} + 2.4}{\sqrt{50L - 2}}
= \frac{2\sqrt{2\ln 480 + 3L} + 2.4}{\sqrt{50L - 2}}.
$$
The last expression is decreasing for $L \ge \ln 50$, and its value at $L = \ln 50$ is less than $0.88$. This proves (1).

For (2), since $r \ge C/40$ and $C \ge 100L$, we have
$
h \ge r \ge \frac{C}{40} \ge \frac{5}{2}L.
$
Therefore,
$
\ln(2k q^{h-2}) = \ln 2 + L - (h-2)M < -\ln 8,
$
which gives $2k q^{h-2} < \frac{1}{8}$.

For (3), put
$
J = \left\lfloor \frac{hM}{4L} \right\rfloor.
$
Since $h \ge (5/2)L$ and $M > 4$, we have $J \ge 2$.
For $1 \le j \le J$, we have
$
\binom{2k}{j} q^j \le (2kq)^j \le k^j = e^{jL}.
$
Thus,
$$
\binom{2k}{j} q^{h+j-2} \le e^{2M} e^{-hM + jL} \le e^{2M} e^{-\frac{3hM}{4}}.
$$
Summing over $j \le J$ and using $J \le k$ gives
$$
\sum_{j=1}^{J} \binom{2k}{j} q^{h+j-2}
\le \exp\left( L + 2M - \frac{3hM}{4} \right)
\le \exp\left( L + 2M - \frac{15LM}{8} \right)
< \frac{1}{16}.
$$
For  $j > J$, since $q < 1/2$, we have
$
\sum_{j > J} e^{\frac{4h}{C}} q^{2j-1} \le 2 e^{\frac{4h}{C}} q^{2J+1}.
$
From
$
2J+1 \ge \frac{hM}{2L} - 1,
$
we get
$$
2 e^{\frac{4h}{C}} q^{2J+1}
\le \exp\left( \ln 2 + \frac{4h}{C} + M - \frac{hM^2}{2L} \right).
$$
Since $M > 4$ and $C \ge 100L$, we have
$
\frac{4}{C} - \frac{M^2}{2L} < 0.
$
Using $h \ge C/40$, the exponent is at most
$$
\ln 2 + M + \frac{1}{10} - \frac{CM^2}{80L}
\le \ln 2 + M + \frac{1}{10} - \frac{5}{4}M^2
< -\ln 16.
$$
Hence,
$$
\sum_{j > J} e^{4h/C} q^{2j-1} < \frac{1}{16}.
$$
The two estimates prove (3).
\end{proof}

\subsection{Proof of Lemma \ref{lem3}}

Let $C = \frac{n}{k}$. Recall that $\mathcal{F}' = \nabla^{2k-1}\mathcal{F}$ and 
$
\mathcal{SD}(\mathcal{F})^{(2k-2)} = \{D\in \mathcal{SD}(\mathcal{F}): |D|=2k-2\}.
$
As in the proof of Lemma \ref{lem1}, we have
 \begin{align}
|\mathcal{F}'| \geq \frac{1}{2n} \binom{n}{2k-1},\qquad \frac{|\mathcal{SD}(\mathcal{F})^{(2k-2)}|}{\binom{n}{2k-2}} \geq 1 - \frac{3}{C}. \label{f42}
 \end{align}
Apply Theorem \ref{concen} to $\mathcal{SD}(\mathcal{F})^{(2k-2)}\subseteq \binom{[n]}{2k-2}$ with $m=n$, $\ell=2k-2$, $\ell'=2k-1$, $t=\lfloor \frac{C}{2} \rfloor-1$, $a=\sqrt{2\ln 8n}$, $\varepsilon=\frac{2a+\sqrt{8\ln 2}}{\sqrt{t}}$, $\alpha=\frac{|\mathcal{SD}(\mathcal{F})^{(2k-2)}|}{\binom{n}{2k-2}}\ge 1-\frac{3}{C}$.
By Lemma \ref{le5} (1), we have
\begin{align*}
 \varepsilon <0.88.
\end{align*}
Together with $C>300$, this gives
$$
\alpha-\epsilon \geq 1-\frac{3}{C}-0.88 > 0.11.
$$
 Let $H$ be distributed uniformly on $\binom{[n]}{2k-1}$. 
By (\ref{f42}), we have $\mathbb P\left[H\in \mathcal{F}'\right]\ge \frac{1}{2n}$.
The concentration inequality gives
   $$\mathbb P\left[\left|\mathcal{SD}(\mathcal{F})^{(2k-2)}\cap \binom{\overline{H}}{2k-2}\right|< 0.11\binom{n-2k+1}{2k-2}\right]< 2e^{-\frac{\mathstrut{a^2}}{2}}=\frac{1}{4n}.$$
	Therefore,  there exists a set $F'\in \mathcal{F}'$ such that $\left|\mathcal{SD}(\mathcal{F})^{(2k-2)}\cap\binom{\overline{F'}}{2k-2}\right|\ge0.11\binom{n-2k+1}{2k-2}$. Without loss of generality,  assume $F'=[2k-1]$.
 Put
$$
m=n-2k+1,\qquad \ell=2k-2,\qquad M_0=\binom{m}{\ell}.
$$

For every $A\in \mathcal{SD}(\mathcal{F})^{(2k-2)}$ with $A\cap F'=\emptyset$, there exists an element $x$ such that $A\cup\{x\}\in \mathcal{F}'$. Since $\mathcal{F}'$ is intersecting and $F'\in \mathcal{F}'$, we have $(A\cup\{x\})\cap F'\neq \emptyset$. Because $A\cap F'=\emptyset$, it follows that $x\in F'$. Assign each such $A$ to one possible $x\in F'$ and denote by $\mathcal{E}_x$ the family assigned to $x$. Then
\begin{align}
\sum_{x\in F'}|\mathcal{E}_x|\geq 0.11M_0. \label{f422}
\end{align}
If $x\neq y$, then $\mathcal{E}_x$ and $\mathcal{E}_y$ are cross intersecting; otherwise $A\in \mathcal{E}_x$, $B\in \mathcal{E}_y$ and $A\cap B=\emptyset$ would imply that $A\cup\{x\}$ and $B\cup\{y\}$ are disjoint members of $\mathcal{F}'$.
Let
$$
p=\binom{m-1}{\ell-1}.
$$
Since
$
\frac{p}{M_0}=\frac{\ell}{m}\leq \frac{2k}{n-2k}=\frac{2}{C-2}<0.01,
$
(\ref{f422}) implies $\sum_{x\in F'}|\mathcal{E}_x|>3p$.

We claim that one of the families $\mathcal{E}_x$ has size at least $\sum_{y\in F'}|\mathcal{E}_y|-p$. Suppose not. Put $S=\sum_{y\in F'}|\mathcal{E}_y|$. Then $S>3p$ and $|\mathcal{E}_x|<S-p$ for every $x$. If some $|\mathcal{E}_x|>p$, then $\{x\}$ and $F'\setminus\{x\}$ give a partition whose two parts have total weights larger than $p$. If all $|\mathcal{E}_x|\leq p$, order the elements of $F'$ arbitrarily and take the first initial segment whose total weight exceeds $p$. Its weight is at most $2p$. Thus, the complementary weight is at least $S-2p>p$. Thus in either case there is a partition $F'=I_1\cup I_2$ such that
$$
\sum_{x\in I_1}|\mathcal{E}_x|>p,\qquad \sum_{x\in I_2}|\mathcal{E}_x|>p.
$$
The two unions $\bigcup_{x\in I_1}\mathcal{E}_x$ and $\bigcup_{x\in I_2}\mathcal{E}_x$ are cross intersecting subfamilies of \(\binom{[n]\setminus F'}{\ell}\), both of size larger than $p=\binom{m-1}{\ell-1}$. This contradicts Theorem \ref{cros}. Therefore, for some $x\in F'$,
\begin{align}
|\mathcal{E}_x| \ge \sum_{y\in F'}|\mathcal{E}_y| - p > (0.11 - 0.01)M_0 = 0.10M_0. \label{f423}
\end{align}

We claim that
\begin{align}
|\mathcal{E}_x| > M_0 - \binom{m-r}{\ell}. \label{f424}
\end{align}
Indeed,
$$
\frac{\binom{m-r}{\ell}}{\binom{m}{\ell}}
= \prod_{i=0}^{\ell-1} \left(1 - \frac{r}{m-i}\right)
\geq 1 - \sum_{i=0}^{\ell-1} \frac{r}{m-i}
\geq 1 - \frac{r\ell}{m-\ell+1}.
$$
Since
$
r \leq \frac{C}{30},\ \ell \leq 2k$ and $ m-\ell+1 = n-4k+4 \geq (C-4)k$,
we have
$
\frac{r\ell}{m-\ell+1} \leq \frac{C}{15(C-4)} < 0.07.
$
Therefore,
$$
M_0 - \binom{m-r}{\ell} < 0.07M_0,
$$
and (\ref{f424}) follows from (\ref{f423}).
We may assume that the index in (\ref{f424}) is $x=1$.  For every $B\subseteq F'$ with $1\notin B$, recall that
$
\mathcal{F}'(B,F')=\{T\subseteq [n]\backslash F': T\cup B\in \mathcal{F}'\}.
$
The families $\mathcal{E}_1$ and $\mathcal{F}'(B,F')$ are cross intersecting. Indeed, if $A\in \mathcal{E}_1$ and $T\in \mathcal{F}'(B,F')$ are disjoint, then $A\cup\{1\}$ and $T\cup B$ would be disjoint members of $\mathcal{F}'$. Since
$
|\mathcal{E}_1| > \binom{m}{\ell} - \binom{m-r}{\ell},
$
Theorem \ref{cros} gives
$$
|\mathcal{F}'(B,F')| \leq \binom{m-r}{m-2k+1+|B|}. 
$$
Summing this over all $B\subseteq F'\setminus\{1\}$, we obtain
$$
\begin{aligned}
|\mathcal{F}'(\overline{1})|
&\leq \sum_{B\subseteq F'\setminus\{1\}} \binom{m-r}{m-2k+1+|B|} \\
&= \sum_{b=0}^{2k-2} \binom{2k-2}{b} \binom{m-r}{m-2k+1+b} \\
&= \binom{n-r-1}{n-2k}.
\end{aligned}
$$
Consequently,
$$
\gamma(\mathcal{F}') \leq \binom{n-r-1}{n-2k}.
$$

\subsection{Proof of Lemma \ref{lem4}}

Let $C = \frac{n}{k}$. Suppose that $\gamma(\mathcal{F}') = |\mathcal{F}'(\overline{1})|$ satisfies
$
1 \le \gamma(\mathcal{F}') \le \binom{n - r}{n - 2k}.
$
Then $|\gamma(\mathcal{F})|\ge 1$, and consequently
$\gamma(\mathcal{F}')\ge\binom{n-k-1}{k-1}=\binom{n-k-1}{n-2k}$.
Assume that $\gamma({\mathcal{F}'})=|\mathcal{F}'(\overline1)|=\binom{x}{n-2k}$ for some $x$ with $n-k-1\le x\le n-r.$
Put
$
h=n-x.
$
Then
$
r\le h\le k+1. 
$

Define
$$
\mathcal{G} = \{[2,n] \setminus F : F \in \mathcal{F}(\overline{1})\}.
$$
Also define
$$
\mathcal{G}' = \{[2,n] \setminus F' : F' \in \mathcal{F}'(\overline{1})\} \subseteq \binom{[2,n]}{n-2k}.
$$
Then
$
|\mathcal{G}'|=\binom{x}{n-2k}.
$
For  $0\le t\le k-1$, define
 $$\mathcal{G}^{k+t}=\{[2, n] \setminus T : T \in (\nabla^{k+t}\mathcal{F})(\overline{1})\}.$$ 
Thus,  $\mathcal{G}^{k+0}=\mathcal{G}$ and $\mathcal{G}^{k+k-1}=\mathcal{G}'$.
 Since $\mathcal{F}'=\nabla^{2k-1}\mathcal{F}$, we have $\mathcal{G}'=\Delta^{n-2k}{\mathcal{G}^{k+t}}$ for every $0 \leq t\leq k-1$.
      Assume that $|\mathcal{G}^{k+t}|=\binom{y}{n-1-k-t}$. By Theorem \ref{L79}, we obtain $|\mathcal{G}'|\geq \binom{y}{n-2k}$.
      Since $|\mathcal{G}'|=\binom{x}{n-2k} $, it follows that $y\leq x$.
Consequently,  $$|\mathcal{G}^{k+t}|\le \binom{x}{n-k-t-1}.$$ 
Moreover, since  $n\ge 4k$ and $\mathcal{G}'=\Delta^{n-2k}{\mathcal{G}^{k+0}}=\Delta^{n-2k}{\mathcal{G}}$,  we have $\Delta^i{\mathcal{G}'}=\Delta^i{\mathcal{G}}$ for $1\le i \le 2k$.
By Theorem \ref{L79}, we have
  $$\Delta_0:=|\Delta^{2k-2}{\mathcal{G}}|=|\Delta^{2k-2}{\mathcal{G}'}|\geq\binom{x}{2k-2}$$
We shall compare all error terms with $\Delta_0$.
 For convenience, we write $\mathcal{SD}=\mathcal{SD(F)}$.
	Observe that $|\mathcal{SD}|=|\mathcal{SD}(1)|+|\mathcal{SD}(\overline{1})|$.

\textbf{ We first consider $\mathcal{SD}(\overline{1})$.}
For the layers below $2k-2$, the trivial bound gives
$$
\sum_{i=0}^{k-2}|\mathcal{SD}(\overline{1})^{(2i)}|
\le \sum_{i=0}^{k-2}\binom{n-1}{2i}. 
$$
The remaining part $\mathcal{SD}^{(2k-2)}(\overline{1})$ may arise from symmetric differences between two sets both in $\mathcal{F}(1)$, or both in $\mathcal{F}(\overline{1})$.

For any $H\in \mathcal{SD}(\mathcal{F}(1))^{(2k-2)}$, we have ${\{1}\} \cup H \in \mathcal{F}'$. If $H\in \Delta^{2k-2}{\mathcal{G}}$, then there exists $F' \in \mathcal{F}'(\overline{1})$ such that $H\cap F'=\emptyset$. This would imply  $({\{1}\} \cup H ) \cap  F'=\emptyset$, contradicting the intersecting property of $\mathcal{F}'$.
        Hence, $H\notin \Delta^{2k-2}{\mathcal{G}}$.
        Consequently, 
	$$|\mathcal{SD}(\mathcal{F}(1))^{(2k-2)}|\le \binom{n-1}{2k-2}-\Delta_0.$$

For any $D\in \mathcal{SD}(\mathcal{F}(\bar{1}))^{(2k-2)}$,  we have  	$|\mathcal{SD}(\mathcal{F}(\overline{1}))^{(2k-2)}|\le (2k-1)|\mathcal{F}'(\overline{1})|=(2k-1)|\mathcal{G}'|.$
It follows that
$$
\frac{|\mathcal{G}'|}{\Delta_0}
\le \frac{\binom{x}{n-2k}}{\binom{x}{2k-2}}
= \frac{\binom{x}{2k-h}}{\binom{x}{2k-2}}.
$$
Here we use the standard convention that $\binom{x}{y}$ for real $y$ is defined via the Gamma function.  
Since $2k-h\le 2k-2$, we obtain
$$
\frac{\binom{x}{2k-h}}{\binom{x}{2k-2}}
= \prod_{s=2k-h+1}^{2k-2}\frac{s}{x-s+1}
\le \left(\frac{2k}{x-2k+3}\right)^{h-2}
\le \left(\frac{3}{C}\right)^{h-2},
$$
where the last inequality follows from $x=n-h\geq n-k-1$ and $C>300$.
By Lemma \ref{le5} (2), we have
$$
(2k-1)|\mathcal{G}'|
\le 2k\left(\frac{3}{C}\right)^{h-2}\Delta_0
< \frac{1}{8}\Delta_0. 
$$
Consequently,
\begin{align}
|\mathcal{SD}(\overline{1})|
\le \sum_{i=0}^{k-1}\binom{n-1}{2i}
-\frac{7}{8}\Delta_0. \label{45}
\end{align}

\textbf{We next consider $\mathcal{SD}(1)$.}
Any set in $\mathcal{SD}(1)$ must be the symmetric difference of one set from $\mathcal{F}(1)$ and one from $\mathcal{F}(\overline{1})$.
Hence, for any $D\in \mathcal{SD}(1)$,  we have $|D|=2k-2j-1$ with $1\le j\le k-1$. 
For any $D \in \mathcal{SD}(1)^{(2k-2j-1)}$, there exist $A \in \mathcal{F}(1)$ and $B \in \mathcal{F}(\overline{1})$ such that $D = A \triangle B$. Then  $|A\cap B|=j$ and  $|A\cup B|=2k-j-1$. Thus, $D$ is contained in some member of $\nabla^{2k-j-1}(\mathcal{F}(\overline{1}))=(\nabla^{2k-j-1}\mathcal{F})(\overline{1})$.  Each member of $\nabla^{2k-j-1}(\mathcal{F}(\overline{1}))$ may contain at most $\binom{2k-j-1}{j}$ such symmetric differences. 
Therefore, 
 \begin{align}
 |\mathcal{SD}(1)^{(2k-2j-1)}|&\le \binom{2k-j-1}{j}|\nabla^{2k-j-1}(\mathcal{F}(\overline{1}))| \notag\\ &=\binom{2k-j-1}{j}|\mathcal{G}^{2k-j-1}|\notag\\
&\le \binom{2k-j-1}{j}\binom{x}{n-2k+j}.\label{46}
 \end{align}
On the other hand, the trivial bound gives
\begin{align}
 |\mathcal{SD}(1)^{(2k-2j-1)}|\le 
 \binom{n-1}{2k-2j-1}.\label{47}
 \end{align}

 We  claim that
\begin{align}
\frac{\binom{n-1}{2k-2j-1}}{\binom{x}{2k-2}}
\le e^{\frac{4h}{C}}\left(\frac{3}{C}\right)^{2j-1}. \label{48}
\end{align}
Indeed,
$
\frac{\binom{n-1}{2k-2j-1}}{\binom{x}{2k-2}}
=
\frac{\binom{n-1}{2k-2}}{\binom{x}{2k-2}}
\cdot
\frac{\binom{n-1}{2k-2j-1}}{\binom{n-1}{2k-2}}.
$
The first factor satisfies
$$
\frac{\binom{n-1}{2k-2}}{\binom{x}{2k-2}}
= \prod_{s=0}^{2k-3}\frac{n-1-s}{x-s}
= \prod_{s=0}^{2k-3}\left(1+\frac{h-1}{x-s}\right)
\le \exp\left(\frac{(2k-2)h}{x-2k+3}\right)
\le e^{\frac{4h}{C}}.
$$
The second factor satisfies
$$
\frac{\binom{n-1}{2k-2j-1}}{\binom{n-1}{2k-2}}
= \prod_{s=0}^{2j-2}\frac{2k-2-s}{n-2k+2+s}
\le \left(\frac{2k}{n-2k}\right)^{2j-1}
\le \left(\frac{3}{C}\right)^{2j-1}.
$$
This proves (\ref{48}).

Using $x=n-h$, we have
$
\binom{x}{n-2k+j}=\binom{x}{x-(n-2k+j)}=\binom{x}{2k-h-j}.
$
It follows that
$$
\frac{\binom{x}{n-2k+j}}{\binom{x}{2k-2}}
=
\frac{\binom{x}{2k-h-j}}{\binom{x}{2k-2}}
= \prod_{s=2k-h-j+1}^{2k-2}\frac{s}{x-s+1}
\le \left(\frac{2k}{x-2k+3}\right)^{h+j-2}
\le \left(\frac{3}{C}\right)^{h+j-2}.
$$
Consequently,
\begin{align}
\binom{2k-j-1}{j}
\frac{\binom{x}{n-2k+j}}{\binom{x}{2k-2}}
\le \binom{2k}{j}\left(\frac{3}{C}\right)^{h+j-2}. \label{49}
\end{align}

From (\ref{46})--\ref{49}) and Lemma \ref{le5} (3), we obtain
\begin{align*}
|\mathcal{SD}(1)|
\le\binom{x}{2k-2}
\sum_{j=1}^{k-1}
\min\left\{
e^{\frac{4h}{C}}\left(\frac{3}{C}\right)^{2j-1},
\binom{2k}{j}\left(\frac{3}{C}\right)^{h+j-2}
\right\}<\frac{1}{8}\Delta_0.
\end{align*}
Combining this with (\ref{45}), we get
\begin{align*}|\mathcal{SD}|&=|\mathcal{SD}(1)|+|\mathcal{SD}(\overline{1})|\\
&< \sum_{i=0}^{k-1}\binom{n-1}{2i}
-\frac{7}{8}\Delta_0 + \frac{1}{8}\Delta_0 \\
&= \sum_{i=0}^{k-1}\binom{n-1}{2i}
-\frac{3}{4}\Delta_0<\sum_{i=0}^{k-1}\binom{n-1}{2i}.
\end{align*}
This completes the proof of Lemma \ref{lem4}.

\section*{Declaration of competing interest}
The authors declare that they have no conflicts of interests to this paper.

\section*{Data availability}
No data was used for the research described in the paper.


\end{document}